\documentclass[reqno,11pt]{amsart}
\usepackage{amssymb}
\usepackage[hidelinks]{hyperref}
\usepackage{latexsym}
\usepackage{bbm}
\usepackage{amssymb}
\usepackage{amsmath}
\usepackage{amsthm}
\usepackage{mathrsfs}
\usepackage{amsbsy}
\usepackage{cite}
\usepackage[latin1]{inputenc}
\usepackage[initials, shortalphabetic]{amsrefs}
\usepackage[left=3.5cm,right=3.5cm,bottom=3cm,top=3cm]{geometry}

\theoremstyle{plain}
\newtheorem{theorem}{Theorem}[section]

\newtheorem{corollary}[theorem]{Corollary}

\newtheorem{lemma}[theorem]{Lemma}

\theoremstyle{definition}
\newtheorem{criterion}[theorem]{Criterion}
\newtheorem{remark}[theorem]{Remark}
\newtheorem{definition}[theorem]{Definition}
\newtheorem*{definition-no}{Definition}

\begin{document}
\title[Nonclassifiability of UHF $L^p$-operator algebras]{Nonclassifiability
of UHF\\
$L^p$-operator algebras}
\author[Eusebio Gardella]{Eusebio Gardella}
\address{Eusebio Gardella\\
Department of Mathematics\\
Deady Hall, University of Oregon\\
Eugene OR 97403-1222, USA\\
and Fields Institute for Research in Mathematical Sciences\\
222 College Street\\
Toronto ON M5T 3J1, Canada.}
\email{gardella@uoregon.edu}
\urladdr{http://pages.uoregon.edu/gardella/}
\author{Martino Lupini}
\address{Martino Lupini\\
Department of Mathematics and Statistics\\
N520 Ross, 4700 Keele Street\\
Toronto Ontario M3J 1P3, Canada, and Fields Institute for Research in
Mathematical Sciences\\
222 College Street\\
Toronto ON M5T 3J1, Canada.}
\email{mlupini@mathstat.yorku.ca}
\urladdr{http://www.lupini.org/}
\thanks{Eusebio Gardella was partially supported by the US National Science
Foundation under Grant DMS-1101742. Martino Lupini was supported by the York
University Susan Mann Dissertation Scholarship. This work was initiated
while the authors were at the Banff International Research Station in
occasion of the workshop ``Dynamics and C*-algebras: Amenability and
Soficity''. The hospitality of the BIRS center is gratefully acknowledged.}
\dedicatory{}

\begin{abstract}
We prove that simple, separable, monotracial UHF $L^{p}$-operator algebras
are not classifiable up to (complete) isomorphism using countable
structures, such as K-theoretic data, as invariants. The same assertion
holds even if one only considers UHF $L^{p}$-operator algebras of tensor
product type obtained from a diagonal system of similarities. For $p=2$, it
follows that separable nonselfadjoint UHF operator algebras are not
classifiable by countable structures up to (complete) isomorphism. Our
results, which answer a question of N.\ Christopher Phillips, rely on Borel
complexity theory, and particularly Hjorth's theory of turbulence.
\end{abstract}

\maketitle

\section{Introduction}

Suppose that $X$ is a standard Borel space and $\lambda $ is a Borel
probability measure on $X$. For $p\in [1,\infty)$, we denote by $L^{p}(
\lambda ) $ the Banach space of Borel-measurable complex-valued functions on 
$X$ (modulo null sets), endowed with the $L^{p}$-norm. Let $B( L^{p}(
\lambda ) ) $ denote the Banach algebra of bounded linear operators on $%
L^{p}( \lambda ) $ endowed with the operator norm. We will identify the
Banach algebra $M_{n}( B( L^{p}( \lambda ) ) ) $ of $n\times n$ matrices
with entries in $B( L^{p}( \lambda ) ) $, with the algebra $B(L^{p}( \lambda
) ^{\oplus n})$ of bounded linear operators on the $p$-direct sum $L^{p}(
\lambda ) ^{\oplus n}$ of $n$ copies of $L^{p}( \lambda )$.

A (concrete)\emph{\ }separable,\emph{\ }unital $L^{p}$\emph{-operator algebra%
}, is a separable, closed subalgebra of $B( L^{p}( \lambda ) ) $ containing
the identity operator. (Such a definition is consistent with \cite[%
Definition 1.1]{phillips_crossed_2013}, in view of \cite[Proposition 1.25]%
{phillips_crossed_2013}.) In the following, all $L^{p}$-operator algebras
will be assumed to be separable and unital. Every unital $L^{p}$-operator
algebra $A\subseteq B(L^p(\lambda))$ is in particular a $p$-operator space
in the sense of \cite[§4]{daws_p-operator_2010}, with matrix norms obtained
by identifying $M_{n}( A) $ with a subalgebra of $M_{n}( B( L^{p}( \lambda )
) ) $. Such algebras have been introduced and studied by N. Christopher
Phillips in 
\cites{phillips_analogs_2012, phillips_simplicity_2013,
phillips_isomorphism_2013}. Many important classes of C*-algebras have been
shown to have $L^{p}$-analogs, including Cuntz algebras \cite%
{phillips_analogs_2012}, UHF algebras \cite{phillips_isomorphism_2013}, AF
algebras \cite{phillips_AF_2014}, and more generally groupoid C*-algebras 
\cite{gardella_representations_2014}.

If $A$ is a unital complex algebra, then an $L^{p}$-\emph{representation} of 
$A$ on a standard Borel probability space $(X,\lambda )$ is a unital algebra
homomorphism $\rho \colon A\rightarrow B(L^{p}(\lambda ))$. The closure
inside $B(L^{p}(\lambda ))$ of $\rho (A)$ is an $L^{p}$-operator algebra,
called the $L^{p}$-\emph{operator algebra associated with} $\rho $. It can
be identified with the completion of $A$ with respect to the operator
seminorm structure $\left\Vert [a_{ij}]\right\Vert _{\rho }=\left\Vert [\rho
(a_{ij})]\right\Vert _{M_{n}(B(L^{p}(\lambda )))}$ for $[a_{ij}]\in M_{n}(A)$%
; see \cite[1.2.16]{blecher_operator_2004}. If $A$ and $B$ are $L^{p}$%
-operator algebras, a \emph{unital homomorphism }$\varphi \colon
A\rightarrow B$ is an algebra homomorphism such that $\varphi (1)=1$. The $n$%
-th \emph{amplification} $\varphi ^{(n)}\colon M_{n}(A)\rightarrow M_{n}(B)$
is defined by $[a_{ij}]\mapsto \lbrack \varphi (a_{ij})]$. A unital
homomorphism $\varphi $ is \emph{completely bounded} if every amplification $%
\varphi ^{(n)}$ is bounded and 
\begin{equation*}
\left\Vert \varphi \right\Vert _{cb}:=\sup_{n\in \mathbb{N}}\left\Vert
\varphi ^{(n)}\right\Vert
\end{equation*}%
is finite.

\begin{definition}
\label{Definition:isomorphism} Let $A$ and $B$ be unital $L^{p}$-operator
algebras.

\begin{enumerate}
\item $A$ and $B$ are said to be \emph{(completely) isomorphic}, if there is
a (completely) bounded unital isomorphism $\varphi \colon A\to B$ with
(completely) bounded inverse $\varphi^{-1}\colon B\to A$.

\item $A$ and $B$ are said to be \emph{(completely) commensurable} if there
are (completely) bounded unital homomorphisms $\varphi\colon A\to B$ and $%
\psi\colon B\to A$.
\end{enumerate}
\end{definition}

For $d\in \mathbb{N}$, we denote by $M_{d}$ the unital algebra of $d\times d$
complex matrices, with matrix units $\left\{ e_{i,j}\right\} _{1\leq i,j\leq
d}$. Let $\boldsymbol{d}=(d_{n})_{n\in \mathbb{N}}$ be a sequence in $%
\mathbb{N}$, and let $\boldsymbol{\rho }=(\rho _{n})_{n\in \mathbb{N}}$ be a
sequence of representations $\rho _{n}\colon M_{d_{n}}\rightarrow
B(L^{p}(X_{n},\lambda _{n}))$. Define $M_{\boldsymbol{d}}$ to be the
algebraic infinite tensor product $\bigotimes\limits_{n\in \mathbb{N}%
}M_{d_{n}}$. Let $X=\prod\limits_{n\in \mathbb{N}}X_{n}$ be the product
Borel space and $\lambda =\bigotimes\limits_{n\in \mathbb{N}}\lambda _{n}$
be the product measure. We naturally regard the algebraic tensor product $%
\bigotimes\limits_{n\in \mathbb{N}}B(L^{p}(\lambda _{n}))$ as a subalgebra
of $B(L^{p}(\lambda ))$. The correspondence%
\begin{align*}
M_{\boldsymbol{d}}& \rightarrow \bigotimes\limits_{n\in \mathbb{N}%
}B(L^{p}(\lambda _{n}))\subseteq B(L^{p}(\lambda )) \\
a_{1}\otimes \cdots \otimes a_{k}& \mapsto \rho _{1}(a_{1})\otimes \cdots
\otimes \rho _{k}(a_{k}),
\end{align*}%
extends to a unital homomorphism $M_{\boldsymbol{d}}\rightarrow
B(L^{p}(\lambda ))$.

\begin{definition}
The algebra $A( \boldsymbol{d},\boldsymbol{\rho }) $ as defined in \cite[%
Example 3.8]{phillips_analogs_2012}, is the $L^{p}$-operator algebra
associated with $\boldsymbol{\rho }$. A \emph{UHF }$L^{p}$\emph{-operator
algebra of tensor product type }$\boldsymbol{d}$ is an algebra of the form $%
A( \boldsymbol{d},\boldsymbol{\rho }) $ for some sequence $\boldsymbol{\rho }
$ as above; see \cite[Definition 3.9]{phillips_analogs_2012} and \cite[%
Definition 1.7]{phillips_isomorphism_2013}.
\end{definition}

A special class of UHF $L^{p}$-operator algebras of tensor product type $%
\boldsymbol{d}$ has been introduced in \cite[Section 5]%
{phillips_isomorphism_2013}. For $d\in \mathbb{N}$, denote by $c_d$ the
normalized counting measure on $d=\{ 0,1,2,\ldots ,d-1\} $, and set $%
\ell^p(d)=L^p(\{0,\ldots,d-1\},c_d)$. The \emph{(canonical) spatial
representation }$\sigma ^{d}$ of $M_{d}$ on $\ell ^{p}( d) $ is defined by
setting%
\begin{equation*}
\left( \sigma ^{d}( a) \xi \right) ( j) =\sum_{i=0,\ldots,d-1}a_{ij}\xi ( i)
\end{equation*}%
for $a\in M_d$, for $\xi\in \ell^p(d)$ and $j=0,\ldots,d-1$; see \cite[%
Definition 7.1]{phillips_analogs_2012}. Observe that the corresponding
matrix norms on $M_{d}$ are obtained by identifying $M_{d}$ with the algebra
of bounded linear operators on $\ell ^{p}( d) $.

Fix a real number $\gamma $ in $[1,+\infty )$, and an enumeration $%
(w_{d,\gamma ,k})_{k\in \mathbb{N}}$ of all diagonal $d\times d$ matrices
with entries in $[1,\gamma ]\cap \mathbb{Q}$. Let $X$ be the disjoint union
of countably many copies of $\{0,1,\ldots ,d-1\}$, and let $\lambda _{d}$ be
the Borel probability measure on $X$ that agrees with $2^{-k}c_{d}$ on the $%
k $-th copy of $\{0,1,\ldots ,d-1\}$. We naturally identify the algebraic
direct sum $\bigoplus\limits_{n\in \mathbb{N}}B(\ell ^{p}(d))$ with a
subalgebra of $B(L^{p}(\lambda _{d}))$. The map%
\begin{align*}
M_{d}& \rightarrow \bigoplus\limits_{n\in \mathbb{N}}B(\ell
^{p}(d))\subseteq B(L^{p}(\lambda _{d})) \\
x& \mapsto \left( \sigma ^{d}\left( w_{d,\gamma ,k}xw_{d,\gamma
,k}^{-1}\right) \right) _{k\in \mathbb{N}}
\end{align*}%
defines a representation $\rho ^{\gamma }\colon M_{d}\rightarrow
B(L^{p}(\lambda _{d}))$.

For a sequence $\boldsymbol{\gamma }$ in $[ 1,+\infty ) $, we will denote by 
$\boldsymbol{\rho }^{\boldsymbol{\gamma }}$ the sequence of representations $%
\rho ^{\gamma _{n}}\colon M_{d_{n}}\to B( L^{p}( \lambda_{d_n} ) ) $
described in the paragraph above. Following the terminology in \cite[Section
3 and Section 5]{phillips_isomorphism_2013}, we say that the corresponding
UHF $L^{p}$-operator algebras $A( \boldsymbol{d},\boldsymbol{\rho }^{%
\boldsymbol{\gamma }}) $ are \emph{obtained from a diagonal system of
similarities }.

\begin{definition}
If $A$ is a unital Banach algebra, a \emph{normalized trace} on $A$ is a
continuous linear functional $\tau \colon A\rightarrow \mathbb{C}$ with $%
\tau (1)=1$, satisfying $\tau (ab)=\tau (ba)$ for all $a,b\in A$. The
algebra $A$ is said to be \emph{monotracial} if $A$ has a unique normalized
trace.
\end{definition}

Recall that a Banach algebra is said to be \emph{simple} if it has no
nontrivial closed two-sided ideals.

\begin{remark}
It was shown in \cite[Theorem 3.19(3)]{phillips_simplicity_2013} that UHF $%
L^p$-operator algebras obtained from a diagonal system of similarities are
always simple and monotracial.
\end{remark}

Problem 5.15 of \cite{phillips_isomorphism_2013} asks to provide invariants
which classify, up to isomorphism, some reasonable class of UHF $L^{p}$%
-operator algebras, such as those constructed using diagonal similarities.
The following is the main result of the present paper.

\begin{theorem}
\label{Theorem:main}The simple, separable, monotracial UHF $L^{p}$-operator
algebras are not classifiable by countable structures up to any of the
following equivalence relations:

\begin{enumerate}
\item complete isomorphism;

\item isomorphism;

\item complete commensurability;

\item commensurability.
\end{enumerate}

The same conclusions hold even if one only considers UHF $L^{p}$-operator
algebras of tensor product type $\boldsymbol{d}$ obtained from a diagonal
system of similarities for a fixed sequence $\boldsymbol{d}=(d_{n})_{n\in 
\mathbb{N}}$ of positive integers such that, for every distinct $n,m\in 
\mathbb{N}$, neither $d_{n}$ divides $d_{m}$ nor $d_{m}$ divides $d_{n}$.
\end{theorem}

It follows from Theorem \ref{Theorem:main} that simple, separable,
monotracial UHF $L^{p}$-operator algebras are not classifiable by
K-theoretic data, even after adding to the K-theory a countable collection
of invariants consisting of countable structures. When $p=2$, Theorem \ref%
{Theorem:main} asserts that separable nonselfadjoint UHF operator algebras
are not classifiable by countable structures up to isomorphism. This
conclusion is in stark constrast with Glimm's classification of UHF
C*-algebras by their corresponding supernatural number \cite%
{glimm_certain_1960}. (Observe that, in view of Glimm's classification,
Banach-algebraic isomorphism and $\ast $-isomorphism coincide for UHF
C*-algebras.)

\section{Borel complexity theory}

In order to obtain our main result, we will work in the framework of Borel
complexity theory. In such a framework, a classification problem is regarded
as an equivalence relation $E$ on a standard Borel space $X$. If $F$ is
another equivalence relation on another standard Borel space $Y$, a \emph{%
Borel reduction} from $E$ to $F$ is a Borel function $g\colon X\rightarrow Y$
with the property that%
\begin{equation*}
xEx^{\prime }\quad \text{if and only if}\quad g(x)Fg(x^{\prime })\text{.}
\end{equation*}

The map $g$ can be seen as a classifying map for the objects of $X$ up to $E$%
. The requirement that $g$ is Borel captures the fact that $g$ is \emph{%
explicit }and \emph{constructible} (and not, for example, obtained by using
the Axiom of Choice). The relation $E$ is \emph{Borel reducible }to $F$ if
there is a Borel reduction from $E$ to $F$. This can be interpreted as
asserting that is it possible to explicitly classify the elements of $X$ up
to $E$ using $F$-classes as invariants.

The notion of Borel reducibility provides a way to compare the complexity of
classification problems in mathematics. Some distinguished equivalence
relations are then used as benchmarks of complexity. The first such
benchmark is the relation $=_{\mathbb{R}}$ of equality of real numbers. (One
can replace $\mathbb{R}$ with any other Polish space.) An equivalence
relation is called \emph{smooth }if it is Borel reducible to $=_{\mathbb{R}}$%
. Equivalently, an equivalence relation is smooth if its classes can be
explicitly parametrized by the points of a Polish space. For instance, the
above mentioned classification of UHF C*-algebras due to Glimm \cite%
{glimm_certain_1960} shows that the classification problem of UHF\
C*-algebras is smooth. Smoothness is a very restrictive notion, and many
natural classification problems transcend such a benchmark. For instance,
the relation of isomorphism of rank $1$ torsion-free abelian groups is not
smooth; see \cite{hjorth_measuring_2002}.

A more generous notion of classifiability is being \emph{classifiable by
countable structures}. Informally speaking, an equivalence relation $E$ on a
standard Borel space $X$ is classifiable by countable structures if it is
possible to explicitly assign to the elements of $X$ complete invariants up
to $E$ that are countable structures, such as as countable (ordered) groups,
countable (ordered) rings, etcetera. To formulate precisely this definition,
let $\mathcal{L}$ be a countable first order language \cite[Definition 1.1.1]%
{marker_model_2002}. The class $\mathrm{Mod}(\mathcal{L})$ of $\mathcal{L}$%
-structures supported by the set $\mathbb{N}$ of natural numbers can be
regarded as a Borel subset of $\prod\limits_{n\in \mathbb{N}}2^{\mathbb{N}%
^{n}}$. As such, $\mathrm{Mod}(\mathcal{L})$ inherits a Borel structure
making it a standard Borel space. Let $\cong _{\mathcal{L}}$ be the relation
of isomorphism of elements of $\mathrm{Mod}(\mathcal{L})$.

\begin{definition}
An equivalence relation $E$ on a standard Borel space is said to be \emph{%
classifiable by countable structures}, if there exists a countable first
order language $\mathcal{L}$ such that $E$ is Borel reducible to $\cong _{%
\mathcal{L}}$.
\end{definition}

The Elliott-Bratteli classification of AF C*-algebras %
\cites{elliott_classification_1976,bratteli_inductive_1972} shows, in
particular, that AF C*-algebras are classifiable by countable structures up
to $\ast $-isomorphism. Any smooth equivalence relation is in particular
classifiable by countable structures.

Many naturally occurring classification problems in mathematics, and
particularly in functional analysis and operator algebras, have recently
been shown to transcend countable structures. This has been obtained for the
relation of unitary conjugacy of irreducible representations and
automorphisms of non type I C*-algebras 
\cites{hjorth_non-smooth_1997,
kerr_turbulence_2010, farah_dichotomy_2012, lupini_unitary_2014}, conjugacy
of ergodic measure-preserving transformations of the Lebesgue measure space~%
\cite{foreman_anti-classification_2004}, conjugacy of automorphisms of $%
\mathcal{Z}$-stable C*-algebras and McDuff II$_{1}$ factors \cite%
{kerr_borel_2014}, unitary conjugacy of unitary and self-adjoint operators 
\cite{kechris_strong_2001}, and isomorphism of von Neumann factors %
\cites{sasyk_classification_2009, sasyk_turbulence_2010}. The main tool
involved in these results is the theory of turbulence developed by Hjorth in 
\cite{hjorth_classification_2000}.

Suppose that $G\curvearrowright X$ is a continuous action of a Polish group $%
G$ on a Polish space $X$. The corresponding orbit equivalence relation $%
E_{G}^{X}$ is the relation on $X$ obtained by setting $xE_{G}^{X}x^{\prime }$
if and only if $x$ and $x^{\prime }$ belong to the same orbit. Hjorth's
theory of turbulence provides a dynamical condition, called \emph{(generic)
turbulence}, that ensures that a Polish group action $G\curvearrowright X$
yields an orbit equivalence relation $E_{G}^{X}$ that is not classifiable by
countable structures. This provides, directly or indirectly, useful criteria
to prove that a given equivalence relation is not classifiable by countable
structures. A prototypical example of turbulent group action is the action
of $\ell ^{1}$ on $\mathbb{R}^{\mathbb{N}}$ by translation. A standard
argument allows one to deduce the following nonclassification criterion from
turbulence of the action $\ell ^{1}\curvearrowright \mathbb{R}^{\mathbb{N}}$
and Hjorth's turbulence theorem \cite[Theorem 3.18]%
{hjorth_classification_2000}; see for example \cite[Lemma 3.2 and Criterion
3.3]{lupini_unitary_2014}.

Recall that a subspace of a topological space is \emph{meager} if it is
contained in the union of countably many closed nowhere dense sets.

\begin{criterion}
\label{Criterion:nonclassification}Suppose that $E$ is an equivalence
relation on a standard Borel space $X$. If there is a Borel map $f\colon [
0,+\infty ) ^{\mathbb{N}} \to X$ such that

\begin{enumerate}
\item $f( \boldsymbol{t}) Ef( \boldsymbol{t}^{\prime }) $ whenever $%
\boldsymbol{t},\boldsymbol{t}^{\prime }\in [ 0,+\infty ) ^{\mathbb{N}}$
satisfy $\boldsymbol{t}-\boldsymbol{t}^{\prime }\in \ell ^{1}$, and

\item the preimage under $f$ of any $E$-class is meager,
\end{enumerate}

then $E$ is not classifiable by countable structures.
\end{criterion}

We will apply such a criterion to establish our main result.

\section{Nonclassification}

Fix a sequence $\boldsymbol{d}=( d_{n}) _{n\in \mathbb{N}}$ of integers such
that for every distinct $n,m \in \mathbb{N}$, neither $d_{n}$ divides $d_{m
} $ nor $d_{m }$ divides $d_{n}$. In particular, this holds if the numbers $%
d_{n}$ are pairwise coprime. The same argument works if one only assumes
that all but finitely many values of $\boldsymbol{d}$ satisfy such an
assumption. We endow $[ 1,+\infty ) ^{\mathbb{N}}$ with the product
topology, and regard it as the parametrizing space for UHF $L^{p}$-operator
algebras of type $\boldsymbol{d} $ obtained from a diagonal system of
similarities, as described in the previous section; see also \cite[Section 3
and Section 5]{phillips_isomorphism_2013}. We therefore regard (complete)
isomorphism and (complete) commensurability of UHF $L^{p}$-operator algebras
of type $\boldsymbol{d}$, obtained from a diagonal system of similarities,
as equivalence relations on $[ 1,+\infty ) ^{\mathbb{N}}$.

For $\boldsymbol{\gamma }\in \lbrack 1,+\infty )^{\mathbb{N}}$, we denote by 
$A^{\boldsymbol{\gamma }}$ the corresponding UHF $L^{p}$-operator algebra.
In the following, we will denote by $\boldsymbol{\gamma }$ and $\boldsymbol{%
\gamma }^{\prime }$ sequences $(\gamma _{n})_{n\in \mathbb{N}}$ and $(\gamma
_{n}^{\prime })_{n\in \mathbb{N}}$ in $[1,+\infty )^{\mathbb{N}}$. For $%
\gamma \in \lbrack 1,+\infty )$, we denote by $M_{d}^{\gamma }$ the $L^{p}$%
-operator algebra structure on $M_{d}$ induced by the representation $\rho
^{\gamma }$ defined in Section~1. The corresponding matrix norms on $%
M_{d}^{\gamma }$ are denoted by $\Vert \cdot \Vert _{\gamma }$. In
particular, when $\gamma =1$ one obtains the matrix norms induced by the
spatial representation $\sigma ^{d}$ of $M_{d}$. The algebra $A^{\boldsymbol{%
\gamma }}$ can be seen as the $L^{p}$-operator tensor product $%
\bigotimes\limits_{n\in \mathbb{N}}^{p}M_{d_{n}}^{\gamma _{n}}$, as defined
in \cite[Definition 1.9]{phillips_simplicity_2013}. (Note that, unlike in 
\cite{phillips_simplicity_2013}, we write the Hölder exponent $p$ as a
superscript in the notation for tensor products.)

\begin{lemma}
\label{Lemma:isomorphic} Let $\boldsymbol{\gamma },\boldsymbol{\gamma }%
^{\prime }\in \lbrack 1,+\infty )^{\mathbb{N}}$ satisfy%
\begin{equation*}
L:=\prod_{n\in \mathbb{N}}\frac{\gamma _{n}}{\gamma _{n}^{\prime }}<+\infty 
\text{.}
\end{equation*}%
Then the identity map on the algebraic tensor product $M_{\boldsymbol{d}%
}=\bigotimes\limits_{n\in \mathbb{N}}M_{d_{n}}$ extends to a completely
bounded unital homomorphism $A^{\boldsymbol{\gamma }}\rightarrow A^{%
\boldsymbol{\gamma }^{\prime }}$, with $\left\Vert \varphi \right\Vert
_{cb}\leq L$. In other words, the matrix norms $\Vert \cdot \Vert _{%
\boldsymbol{\gamma }}$ and $\Vert \cdot \Vert _{\boldsymbol{\gamma ^{\prime }%
}}$ on the algebraic tensor product $\bigotimes\limits_{n\in \mathbb{N}%
}M_{d_{n}}$ satisfy 
\begin{equation*}
\Vert \cdot \Vert _{\boldsymbol{\gamma }^{\prime }}\leq L\Vert \cdot \Vert _{%
\boldsymbol{\gamma }}.
\end{equation*}
\end{lemma}

\begin{proof}
For $j\in \mathbb{N}$, let $L_{j}=\frac{\gamma _{j}}{\gamma _{j}^{\prime }}$%
. Fix $\varepsilon >0$. In order to prove our assertion, it is enough to
show that if $k\in \mathbb{N}$ and $x$ is an element of $M_{k}\left(
\bigotimes\limits_{j\in \mathbb{N}}M_{d_{j}}\right) $, then $\left\Vert
x\right\Vert _{\boldsymbol{\gamma }^{\prime }}\leq (1+\varepsilon )L\Vert
x\Vert _{\boldsymbol{\gamma }}$. Let $x\in M_{k}\left(
\bigotimes\limits_{j\in \mathbb{N}}M_{d_{j}}\right) $, and choose $n,m\in 
\mathbb{N}$ and $X_{i,j}\in M_{k}(M_{d_{i}})$ for $1\leq i\leq n$ and $1\leq
j\leq m$, satisfying 
\begin{equation*}
x=\sum_{1\leq j\leq m}X_{1,j}\otimes \cdots \otimes X_{n,j}
\end{equation*}%
By definition of the matrix norms on $A_{\boldsymbol{\gamma }}$, for $1\leq
i\leq n$ there exists a diagonal matrix $w_{i}\in M_{d_{i}}$ with entries in 
$[1,\gamma _{i}]$ such that, if $W_{i}\in M_{k}(M_{d_{i}})$ is the diagonal
matrix with entries in $M_{d_{i}}$, and nonzero entries equal to $w_{i}$ (in
other words, $W_{i}=1_{M_{k}}\otimes w_{i}$), then%
\begin{equation*}
\Vert x\Vert _{\boldsymbol{\gamma }}\leq (1+\varepsilon )\left\Vert
\sum_{1\leq j\leq m}W_{1}X_{1,j}W_{1}^{-1}\otimes \cdots \otimes
W_{n}X_{n,j}X_{n}^{-1}\right\Vert \text{.}
\end{equation*}

For $1\leq i\leq n$, we denote the diagonal entries of $w_{i}\in M_{d_{j}}$
by $a_{i,\ell }$, for $\ell =1,\ldots ,d_{i}$. We will define two other
diagonal matrices 
\begin{equation*}
w_{i}^{\prime }=\mathrm{diag}(a_{i,1}^{\prime },\ldots ,a_{i,d_{i}}^{\prime
})\ \mbox{ and }\ r_{i}=\mathrm{diag}(r_{i,1},\ldots ,r_{i,d_{i}})
\end{equation*}%
in $M_{d_{i}}$, with entries in $[1,\gamma _{i}^{\prime }]$ and $[1,L_{i}]$,
respectively, as follows. For $1\leq \ell \leq d_{i}$, we set 
\begin{equation*}
a_{i,\ell }^{\prime }=\left\{ 
\begin{array}{lll}
a_{i,\ell }, & \hbox{if $a_{i,\ell}< \gamma^{\prime}_i$;} &  \\ 
\gamma _{i}^{\prime }, & \hbox{if $a_{j,\ell}\geq \gamma^{\prime}_i$.} & 
\end{array}%
\right.
\end{equation*}%
and 
\begin{equation*}
r_{i,\ell }=\left\{ 
\begin{array}{lll}
1, & \hbox{if $a_{i,\ell}< \gamma^{\prime}_i$;} &  \\ 
\frac{1}{\gamma _{i}^{\prime }}a_{i,\ell }, & 
\hbox{if $a_{i,\ell}\geq
\gamma^{\prime}_i$.} & 
\end{array}%
\right.
\end{equation*}%
Observe that $r_{i,\ell }$ belongs to $[1,L_{i}]$ (since $a_{i,\ell }\leq
\gamma _{i}\leq L_{i}\gamma _{i}^{\prime }$), and that $a_{i,\ell }^{\prime
} $ belongs to $[1,\gamma _{i}^{\prime }]$ for all $1\leq i\leq n$ and $%
1\leq \ell \leq d_{i}$.

Define $w_{i}^{\prime }$ and $r_{i}$ to be the diagonal $d_{i}\times d_{i}$
matrices with diagonal entries $a_{i,\ell }^{\prime }$ and $r_{i,\ell }$ for 
$1\leq \ell \leq d_{i}$. Let $W_{i}^{\prime },R_{i}\in M_{k}(M_{d_{i}})$ be
the diagonal $k\times k$ matrices with entries in $M_{d_{i}}$ having
diagonal entries equal to, respectively, $w_{i}^{\prime }$ and $r_{i}$. (In
other words, $W_{i}^{\prime }=1_{M_{k}}\otimes w_{i}^{\prime }$ and $%
R_{i}=1_{M_{k}}\otimes r_{i}$.)

Then $W_{i}=R_{i}W_{i}^{\prime }$ for all $1\leq i\leq n$. Additionally, 
\begin{equation*}
\left\| R_{i}\right\| \leq L_{i} \ \mbox{ and } \ \left\| R_{i}^{-1}\right\|
\leq 1.
\end{equation*}
Therefore, 
\begin{align*}
\left\Vert x\right\Vert _{\boldsymbol{\gamma }} &\leq
(1+\varepsilon)\left\Vert \sum_{1\leq j\leq m}W_{1}X_{1,j}W_{1}^{-1}\otimes
\cdots \otimes W_{n}X_{n,j}W_{n}^{-1}\right\Vert \\
&=(1+\varepsilon)\left\Vert \sum_{1\leq j\leq m}R_{1}W_{1}^{\prime
}X_{1,j}W_{1}^{^{\prime }-1}R_{1}^{-1}\otimes \cdots \otimes
R_{n}W_{n}^{\prime }X_{n,j}W_{n}^{^{\prime }-1}R_{n}^{-1}\right\Vert \\
&\leq (1+\varepsilon)\left\Vert R_{1}\right\Vert \left\Vert R_{2}\right\Vert
\cdots \left\Vert R_{n}\right\Vert \left\Vert \sum_{1\leq j\leq
m}W_{1}^{\prime }X_{1,j}W_{1}^{\prime -1}\otimes \cdots \otimes
W_{n}^{\prime }X_{n,j}W_{n}^{\prime -1}\right\Vert \\
&\leq (1+\varepsilon)L_{1}\cdots L_{n}\left\Vert \sum_{1\leq j\leq
m}W_{1}^{\prime }X_{1,j}W_{1}^{\prime -1}\otimes \cdots \otimes
W_{n}^{\prime }X_{n,j}W_{n}^{\prime -1}\right\Vert \\
&\leq (1+\varepsilon)L\left\Vert x\right\Vert _{\boldsymbol{\gamma }^{\prime
}}\text{.}
\end{align*}
This concludes the proof.
\end{proof}

\begin{corollary}
\label{Corollary:isomorphic} If $\boldsymbol{\gamma },\boldsymbol{\gamma }%
^{\prime }\in \lbrack 1,+\infty )^{\mathbb{N}}$ satisfy%
\begin{equation*}
\prod_{n\in \mathbb{N}}\max \left\{ \frac{\gamma _{n}}{\gamma _{n}^{\prime }}%
,\frac{\gamma _{n}^{\prime }}{\gamma _{n}}\right\} <+\infty \text{,}
\end{equation*}%
then $A^{\mathbf{\gamma }}$ and $A^{\mathbf{\gamma }^{\prime }}$ are
completely isomorphic.
\end{corollary}

The following lemma can be proved in the same way as \cite[Lemma 5.11]%
{phillips_isomorphism_2013} with the extra ingredient of \cite[Lemma 5.8]%
{phillips_isomorphism_2013}. As before, we denote by $\otimes^{p}$ the $%
L^{p} $-operator tensor product; see \cite[Definition 1.9]%
{phillips_simplicity_2013}.

\begin{lemma}[Phillips]
\label{Lemma:perturbation} Let $L>0$ and let $d\in \mathbb{N}$. Then there
is a constant $R( L,d) >0$ such that the following holds. Whenever $A$ is a
unital $L^{p}$-operator algebra, whenever $\gamma ,\gamma ^{\prime }\in
\lbrack 1,+\infty )$ satisfy%
\begin{equation*}
\gamma ^{\prime }\geq R( L,d) \gamma \text{,}
\end{equation*}%
and $\varphi \colon M_{d}^{\gamma }\to M_{d}^{\gamma ^{\prime }}\otimes
^{p}A $ is a unital\emph{\ }homomorphism with $\|\varphi\| \leq L$, there
exists a unital homomorphism $\psi \colon M_{d}^{\gamma }\to A$ with $%
\|\psi\| \leq L+1$.
\end{lemma}

Our assumption on the values of $\boldsymbol{d}$ will be used for the first
time in the next lemma, where it is shown that sufficiently different
sequences yield noncommensurable UHF $L^{p}$-operator algebras.

The $K_0$-group of a Banach algebra $A$ is defined using idempotents in
matrices over $A$, and a suitable equivalence relation involving
similarities of such idempotents. We refer the reader to \cite[Chapters 5,8,9%
]{blackadar_K-theory_1998} for the precise definition and some basic
properties. What we will need here is the following:

\begin{remark}
\label{remark:K_0} For $n\in\mathbb{N}$ and a unital Banach algebra $A$, if
there exists a unital, continuous homomorphism $M_n\to A$, then the class of
unit of $A$ in $K_0(A)$ must be divisible by $n$.
\end{remark}

\begin{lemma}
Suppose that $\boldsymbol{\gamma },\boldsymbol{\gamma }^{\prime }\in [
1,+\infty ) ^{\mathbb{N}}$ satisfy $\gamma _{n}^{\prime }\geq R( n,d_{n})
\gamma _{n}$ for infinitely many $n\in \mathbb{N}$. Then there is no
continuous unital homomorphism $\varphi \colon A^{\boldsymbol{\gamma }}\to
A^{\gamma ^{\prime }}$.
\end{lemma}

\begin{proof}
Assume by contradiction that $\varphi \colon A^{\boldsymbol{\gamma }}\to A^{%
\boldsymbol{\gamma }^{\prime }}$ is a continuous unital homomorphism and set 
$L=\|\varphi\| $. Pick $n\in \mathbb{N}$ such that $n\geq L$ and $\gamma
_{n}^{\prime }\geq R( n,d_{n}) \gamma _{n}$. Set 
\begin{equation*}
A= \bigotimes\limits^{p}_{m\in\mathbb{N}, m\neq n}M_{d_{m}}^{\gamma _{m}}.
\end{equation*}
Apply Lemma \ref{Lemma:perturbation} to the unital homomorphism $\varphi
\colon M_{d_{n}}^{\gamma _{n}}\to M_{d_{n}}^{\gamma _{n}}\otimes^{p}A$, to
get a unital homomorphism $\psi \colon M_{d_{n}}^{\gamma _{n}}\to A$ with $%
\| \psi \| \leq L+1$.

Using Remark \ref{remark:K_0}, we conclude that the class of the unit of $A$
in $K_{0}(A)$ is divisible by $d_{n}$. On the other hand, the $K$-theory of $%
A$ is easy to compute using that $K$-theory for Banach algebras commutes
with direct limits (with contractive maps). We get%
\begin{equation*}
K_{0}\left( A\right) =\mathbb{Z}\left[ \frac{1}{b}:b\neq 0\text{ divides }%
d_{m}\text{ for some }m\neq n\right]
\end{equation*}%
%
%
%
%
with the unit of $A$ corresponding to $1\in K_{0}(A)\subseteq \mathbb{Q}$.

Since there is a prime appearing in the factorization of $d_{n}$ that does
not divide any $d_{m}$, for $m\neq n$, we deduce that the class of the unit
of $A$ in $K_{0}(A)$ cannot be divisible by $d_{n}$. This contradiction
shows that there is no continuous unital homomorphism $\varphi \colon A^{%
\boldsymbol{\gamma }}\rightarrow A^{\boldsymbol{\gamma }^{\prime }}$
\end{proof}

We say that a set is \emph{comeager} if its complement is meager. Observe
that, by definition, a nonmeager set interescts every comeager set. Recall
that we regard $[1,+\infty )^{\mathbb{N}}$ as the parametrizing space of the
UHF $L^{p}$-operator algebras of tensor product type $\boldsymbol{d}$
obtained from a diagonal system of similarities. Consistently, we regard
(complete) isomorphism and commensurability of such algebras as equivalence
relations on $[1,+\infty )^{\mathbb{N}}$.

\begin{proof}[Proof of Theorem \protect\ref{Theorem:main}]
By \cite[Theorem 3.19(3)]{phillips_simplicity_2013}, every UHF $L^{p}$%
-operator algebra of tensor product type $\boldsymbol{d}$ obtained from a
diagonal system of similarities is simple and monotracial. Therefore, it is
enough to prove the second assertion of Theorem \ref{Theorem:main}. For $%
\boldsymbol{t}\in \lbrack 0,+\infty )^{\mathbb{N}}$, define $\exp (%
\boldsymbol{t})$ to be the sequence $(\exp (t_{n}))_{n\in \mathbb{N}}$ of
real numbers in $[1,\infty )$. By Corollary \ref{Corollary:isomorphic}, if $%
\boldsymbol{t},\boldsymbol{t}^{\prime }\in \lbrack 0,+\infty )^{\mathbb{N}}$
satisfy $\boldsymbol{t}\mathbf{-}\boldsymbol{t}^{\prime }\in \ell ^{1}$,
then $A^{\exp (\boldsymbol{t})}$ and $A^{\exp (\boldsymbol{t}^{\prime })}$
are completely isomorphic. We claim that for any nonmeager subset $C$ of $%
\left[ 0,+\infty \right) ^{\mathbb{N}}$ one can find $\boldsymbol{t},%
\boldsymbol{t}^{\prime }\in C$ such that $A^{\exp (\boldsymbol{t})}$ and $%
A^{\exp (\boldsymbol{t}^{\prime })}$ are not commensurable. This fact
together with Corollary \ref{Corollary:isomorphic} will show that the Borel
function%
\begin{align*}
\lbrack 0,+\infty )^{\mathbb{N}}& \rightarrow \lbrack 1,+\infty )^{\mathbb{N}%
} \\
\boldsymbol{t}& \mapsto \exp (\boldsymbol{t})
\end{align*}%
satisfies the hypotheses of Criterion \ref{Criterion:nonclassification} for
any of the equivalence relations $E$ in the statement of Theorem \ref%
{Theorem:main}, yielding the desired conclusion.

Let then $C$ be a nonmeager subset of $[0,+\infty )^{\mathbb{N}}$, and fix $%
\boldsymbol{t}\in C$. We want to find $\boldsymbol{t}^{\prime }\in C$ such
that $A^{\exp (\boldsymbol{t})}$ and $A^{\exp (\boldsymbol{t}^{\prime })}$
are not commensurable. The set

\begin{align*}
& \left\{ \boldsymbol{t}^{\prime }\in \lbrack 0,+\infty )^{\mathbb{N}}\colon 
\mbox{
for all but finitely many }n\in \mathbb{N}\text{, }\exp (t_{n}^{\prime
})\leq R(n,d_{n})\exp (t_{n})\right\} \\
& \ \ \ \ \ \ \ \ \ \ =\bigcup_{k\in \mathbb{N}}\left\{ \boldsymbol{t}%
^{\prime }\in \lbrack 0,+\infty )^{\mathbb{N}}\colon {}\forall n\geq k\text{%
, }\exp (t_{n}^{\prime })\leq R(n,d_{n})\exp (t_{n})\right\}
\end{align*}%
is a countable union of closed nowhere dense sets, hence meager. Therefore,
its complement

\begin{equation*}
\left\{ \boldsymbol{t}^{\prime }\in \lbrack 0,+\infty )^{\mathbb{N}}\colon %
\mbox{ for infinitely many }n\in \mathbb{N}\text{, }\exp (t_{n}^{\prime
})>R(n,d_{n})\exp (t_{n})\right\} \text{,}
\end{equation*}%
is comeager. In particular, since $C$ is nonmeager, there is $\boldsymbol{t}%
^{\prime }\in C$ such that $\exp (t_{n}^{\prime })\geq R(n,d_{n})\exp
(t_{n}) $ for infinitely many $n\in \mathbb{N}$. By Lemma \ref%
{Lemma:perturbation}, there is no continuous unital homomorphism from $%
A^{\exp (\boldsymbol{t})}$ to $A^{\exp (\boldsymbol{t}^{\prime })}$.
Therefore $A^{\exp (\boldsymbol{t})}$ and $A^{\exp (\boldsymbol{t}^{\prime
})}$ are not commensurable. This concludes the proof of the above claim.
\end{proof}

\bibliographystyle{amsplain}
\bibliography{UHF-library2}

\end{document}